\documentclass[12pt, a4paper]{amsart}
\usepackage{amsmath}
\usepackage{geometry,amsthm,graphics,tabularx,amssymb,shapepar}
\usepackage{verbatim}
\usepackage{amscd}
\usepackage{bm}
\usepackage{xcolor}
\usepackage[all]{xypic}

\newcommand{\fn}{\mathfrak n}

\newcommand{\fh}{\mathfrak h}

\newcommand{\ot}{\otimes}

\newcommand{\CL}{\mathcal{L}}

\newcommand{\C}{\mathbb{C}}

\newcommand{\N}{\mathbb N}
\newcommand{\Z}{\mathbb Z}

\newcommand{\ba}{\begin {eqnarray}}
\newcommand{\ea}{\end {eqnarray}}
\newcommand{\baa}{\begin {eqnarray*}}
\newcommand{\eaa}{\end {eqnarray*}}
\newcommand{\be}{\begin {equation}}
\newcommand{\ee}{\end {equation}}
\newcommand{\bee}{\begin {equation*}}
\newcommand{\eee}{\end {equation*}}

\newcommand{\te}[1]{\textnormal{{#1}}}

\theoremstyle{Theorem}

\theoremstyle{Theorem}

\newtheorem{thm}{Theorem}[section]
\newtheorem{cort}[thm]{Corollary}
\newtheorem{lemt}[thm]{Lemma}
\newtheorem{prpt}[thm]{Proposition}
\newtheorem{thmt}[thm]{Theorem}

\theoremstyle{Theorem}

\theoremstyle{Theorem}

\theoremstyle{Plain}

\theoremstyle{Definition}

\newtheorem{dfnt}[thm]{Definition}

\def\({\left(}

\def\){\right)}

\def \<{{\langle}}
\def \>{{\rangle}}

\allowdisplaybreaks

\numberwithin{equation}{section}
\title[Local Weyl modules and skew Howe duality]{Local Weyl modules and skew Howe duality}

\author{Fulin Chen$^1$}
\address{School of Mathematical Sciences, Xiamen University,
 Xiamen, China 361005}
  \email{chenf@xmu.edu.cn}
\thanks{1. Partially supported by China NSF grant (No. 12471029) and Xiamen NSF grant (No. 3502Z202473005)}
\author{Xin Huang$^2$}
\address{School of Mathematics, Hefei University of Technology, Hefei, China 230009, Xiamen, China 361005}
 \email{xinhuang@hfut.edu.cn}
\thanks{2. Partially supported by China NSF grant (No. 12526548)}
 \author{Siyi Niu}
  \address{School of Mathematical Sciences, Xiamen University,
 Xiamen, China 361005}
  \email{niusiyi@stu.xmu.edu.cn}
\author{Shaobin Tan$^3$}
\address{School of Mathematical Sciences, Xiamen University,
 Xiamen, China 361005}
 \email{tans@xmu.edu.cn}
 \thanks{3. Partially supported by China NSF grant (No. 12131018)}
 
\begin{document}

\begin{abstract}
The  skew $(\mathfrak{gl}_{n}, \mathfrak{gl}_{r})$ Howe duality  states that the exterior algebra $\Lambda(\mathbb{C}^{nr})$ admits a multiplicity-free decomposition under the natural actions of $\mathfrak{gl}_{n}\times \mathfrak{gl}_{r}$.  In this paper, 
by using certain  Lagrange interpolation polynomials of degree $r-1$, we extend the action of $\mathfrak{gl}_{n}$ on $\Lambda(\C^{nr})$  to its loop algebra $L(\mathfrak{gl}_{n})$. 
View $\Lambda(\C^{nr})$ as a module for the loop algebra $L(\mathfrak{sl}_{n})$ of $\mathfrak{sl}_{n}$ by taking restriction.  We prove that every highest weight vector of $\mathfrak{gl}_{n}\times \mathfrak{gl}_{r}$ in $\Lambda(\C^{nr})$ generates a  local Weyl module of $L(\mathfrak{sl}_{n})$. Furthermore, we obtain in this way an explicit realization of all local Weyl modules for $L(\mathfrak{sl}_{n})$.
 \end{abstract}
\maketitle

\section{Introduction}
Motivated by the study of finite-dimensional modules of quantum affine algebras, Chari and Pressley introduced in \cite{CP}
 the notion of local Weyl modules for loop algebras. They conjectured therein that local Weyl modules are the classical limit of standard irreducible modules for quantum affine algebras, and further showed that this conjecture would follow from an additional conjecture concerning the dimensions of local Weyl modules. 
 The dimension conjecture has been proved via algebraic methods in \cite{CP,CL,FL} for the simply laced case  and in  \cite{N} for the non-simply laced case. These works also revealed close connections between local Weyl modules and several other important classes of modules, including Demazure modules and Kirillov--Reshetikhin modules.
  As pointed out by Nakajima, the dimension conjecture can also be deduced from  some deep geometric results in \cite{BN} and \cite{K}.
  
The notion of (local and global) Weyl modules has since been extended to various other interesting algebras, such as current Lie algebras \cite{CFK, FL, N}, twisted loop algebras \cite{CFS}, toroidal Lie algebras \cite{MPS},  and loop  superalgebras \cite{CLS}. On the other hand, Feigin and Makedonskyi introduced generalized Weyl modules in  \cite{FM}  by extending the notion of Weyl modules to arbitrary integral weights, and the connection between these modules and nonsymmetric Macdonald polynomials was investigated in \cite{FMO}.
 
 

Although local Weyl modules arise as the classical limits of irreducible modules for quantum affine algebras, they are generally not irreducible and possess rather intricate structures. 
Several fundamental problems concerning local Weyl modules remain open, most notably their realization problem. The main goal of this paper is to provide a unified construction of local Weyl modules for the loop algebra $\CL(\mathfrak{sl}_n)$ of $\mathfrak{sl}_n$. 
Our construction is motivated by the so-called  skew $(\mathfrak{gl}_{n},\mathfrak{gl}_{r})$ Howe duality on $\Lambda(\C^{nr})$, which yields an explicit realization of all finite-dimensional irreducible $\mathfrak{sl}_{n}$-modules.

In representation theory, one of the prominent notions is that of dual pair for reductive Lie groups/algebras, first introduced by Howe \cite{H1}. This concept has found numerous applications in invariant theory, automorphic representations, quantum groups,  infinite-dimensional Lie algebras, and  related areas. 
In this paper, we focus on the skew Howe duality for complex general linear Lie algebras.
Specifically, it is known that  $(\mathfrak{gl}_{n},\mathfrak{gl}_{r})$ forms a reductive dual pair, namely a pair of maximal commuting reductive subalgebras, inside $\mathfrak{gl}_{nr}$. The natural action of $\mathfrak{gl}_{nr}$ on the exterior algebra $\Lambda(\C^{nr})$ of $\C^{nr}$ therefore induces a $\mathfrak{gl}_{n}\times \mathfrak{gl}_{r}$-action  on $\Lambda(\C^{nr})$. 
The  skew $(\mathfrak{gl}_{n},\mathfrak{gl}_{r})$ Howe duality states that  the exterior algebra $\Lambda(\C^{nr})$  admits a remarkable  multiplicity-free decomposition under this $\mathfrak{gl}_{n}\times \mathfrak{gl}_{r}$-action. 
This decomposition is determined by those highest weight vectors of $\mathfrak{gl}_{n}\times \mathfrak{gl}_{r}$ in $\Lambda(\C^{nr})$, which are given explicitly by Grassmann monomials  (see \cite{H2} or Proposition \ref{jointh}).

Motivated by the work of \cite{CGT},
we extend the action of $\mathfrak{gl}_{n}$ on $\Lambda(\C^{nr})$ to its loop algebra $L(\mathfrak{gl}_{n})$, in terms of  finite sums of partial differential operators
glued together using certain Lagrange interpolation polynomials of degree $r-1$. Restricting this action, we may regard 
the exterior algebra $\Lambda(\C^{nr})$ as a module for the loop algebra $L(\mathfrak{sl}_{n})$ of $\mathfrak{sl}_{n}$. As the main result of this paper, we prove that every highest weight vector of $\mathfrak{gl}_{n}\times \mathfrak{gl}_{r}$ in $\Lambda(\C^{nr})$ generates a  local Weyl module of $L(\mathfrak{sl}_{n})$ (see Theorem \ref{th:main}).
Furthermore, by the factorization property of local Weyl modules, we obtain an explicit realization of  all local Weyl modules for $L(\mathfrak{sl}_{n})$ (see Corollary \ref{cor:main}).

The proof of  Theorem \ref{th:main} is based on a realization of certain local Weyl modules for the 
 $(r-1)$-th Takiff algebra $T_r(\mathfrak{sl}_n)=\mathfrak{sl}_{n}\otimes \C[t]/t^r\C[t]$ of $\mathfrak{sl}_n$.
 Local Weyl modules for Takiff algebras attached to finite-dimensional simple Lie algebras  were introduced and studied in \cite{FMM}. It is also known that every local Weyl module for a loop algebra descends naturally to a module over an appropriate Takiff algebra of the same type (see \cite{CFS}).
 
A key observation is that there is a natural embedding from the 
Takiff algebra $T_r(\mathfrak{gl}_n)=\mathfrak{gl}_{n}\otimes \C[t]/t^r\C[t]$ into $\mathfrak{gl}_{nr}$ (see \cite[Section 2]{L} for example).
Consequently, the natural action of $\mathfrak{gl}_{nr}$ on $\Lambda(\C^{nr})$ induces an action of  $T_r(\mathfrak{gl}_n)$, and hence also of  $T_r(\mathfrak{sl}_n)$, on $\Lambda(\C^{nr})$.
We prove in Theorem \ref{th:tak} that the $T_r(\mathfrak{sl}_n)$-submodule of  $\Lambda(\C^{nr})$ generated by a highest weight vector of $\mathfrak{gl}_{n}\times \mathfrak{gl}_{r}$ is a  local Weyl module. 
  On the other hand, 
we show that the action of  $L(\mathfrak{sl}_{n})$ on $\Lambda(\C^{nr})$ descends to an action of $T_r(\mathfrak{sl}_n)$. 
Moreover, up to twisting by an automorphism, these two 
$T_r(\mathfrak{sl}_n)$-module structures on $\Lambda(\C^{nr})$ coincide, thereby implying Theorem  \ref{th:main}.

Now we give an outline of the structure of this paper. In Section \ref{sec:2}, we review the  skew Howe duality for general linear Lie algebras. In Section \ref{sec:3}, we recall some basics on local Weyl modules for $L(\mathfrak{sl}_{n})$, and  then
 state the main result in Section \ref{sec:4}. In Section \ref{sec:5}, we construct certain local Weyl modules for the $(r-1)$-th Takiff algebra $T_r(\mathfrak{sl}_n)$. Finally, in Section \ref{sec:6}, we prove the main result using the result established in Section \ref{sec:5}.

All the Lie algebras considered in this paper are over the field $\C$ of complex
numbers. We denote the sets of integers, non-negative integers, positive integers,
complex numbers and non-zero complex numbers by $\Z$, $\N$, $\Z_{+}$, $\C$ and $\C^{\times}$, respectively. And, for a Lie algebra $\mathfrak{t}$, we will use the notation $\mathcal{U}(\mathfrak{t})$ to stand for the
universal enveloping algebra of $\mathfrak{t}$. 

\section{skew $(\mathfrak{gl}_n,\mathfrak{gl}_r)$ Howe duality}\label{sec:2}
 In this section, let $n$ and $r$ be two positive integers. Here we recall the skew Howe duality for general linear Lie algebras from \cite{H2}.

Write $\mathfrak{gl}_{n}$ for the complex general linear Lie algebra of rank $n$, which consists of all $n\times n$-matrices over $\C$. For every $1\leq i,j\leq n$, write $E_{i,j}$ for the  $n\times n$-matrix  with $1$ in the $(i,j)$-entry and zeros elsewhere. Note that there is a natural triangular decomposition \begin{align}\label{trigl}
\mathfrak{gl}_{n}=\mathfrak{n}_{n}^+\oplus \mathfrak{d}_n\oplus \mathfrak{n}_{n}^-,
\end{align} of $\mathfrak{gl}_{n}$,
where
\[\mathfrak{n}^{\pm}_{n}:=\sum_{1\le i,j\le n; \pm (j-i)>0} \C E_{i,j}\quad \text{and}\quad \mathfrak{d}_n:=\sum_{i=1}^{n}\C E_{i,i}.
\]

As usual, by a partition, we mean an infinite sequence 
\[\xi=(\xi_1,\xi_2,\dots)\]
of non-negative integers in non-increasing  order and containing only finitely many non-zero terms. The number of non-zero terms in $\xi$ is called its length, denoted by $\ell(\xi)$.
  For every partition $\xi=(\xi_1,\xi_2,\dots)$ with $\ell(\xi)\leq n$,  denote by $L_{\mathfrak{gl}_{n}}(\xi)$ the irreducible highest weight $\mathfrak{gl}_{n}$-module generated by a (non-zero) highest weight vector $v_{\xi}$ such that 
\[\fn_{n}^+.v_{\xi}=0\quad \textrm{and}\quad E_{j,j}.v_{\xi}=\xi_j v_{\xi}\quad \text{for}\ 1\le j\le n.\]


Form the Fock space
\begin{align}\label{fs}
V_{n,r}:=\Lambda[y_{i,j}\mid 1\leq i\leq n,~0\leq j\leq r-1],
\end{align}
which is the exterior algebra in the variables $y_{i,j}$. 
By identifying $V_{n,r}$ with the exterior algebra $\Lambda(\C^{nr})$ of $\C^{nr}$, 
 the natural action of $\mathfrak{gl}_{nr}$ on $\C^{nr}$ (and hence on $\Lambda(\C^{nr})$) induces a $\mathfrak{gl}_{nr}$-module structure on $V_{n,r}$ with the action given by 
\begin{align}\label{acglnr}
    E_{i+kn,j+tn}\mapsto y_{i,k}\frac{\partial}{\partial y_{j,t}}\quad \text{for $1\leq i,j\leq n$ and $0\leq k,t\leq r-1$.}
\end{align}
It is well known that $(\mathfrak{gl}_{n},\mathfrak{gl}_{r})$ forms a dual pair in $\mathfrak{gl}_{nr}$ in the sense that:  $(\mathfrak{gl}_{n},\mathfrak{gl}_{r})$ is a pair of maximal commuting subalgebras in  $\mathfrak{gl}_{nr}$
through the embedding
\begin{align*}
    &\mathfrak{gl}_{n}\hookrightarrow \mathfrak{gl}_{nr},\quad E_{i,j}\mapsto \sum_{k=0}^{r-1}E_{i+kn,j+kn},\\
    &\mathfrak{gl}_{r}\hookrightarrow \mathfrak{gl}_{nr},\quad E_{k,t}\mapsto \sum_{\ell=1}^{n}E_{\ell+(k-1)n,\ell+(t-1)n},
\end{align*}
where $1\leq i,j\leq n$ and $1\leq k,t\leq r$. This, together with \eqref{acglnr}, provides a $(\mathfrak{gl}_{n}\times \mathfrak{gl}_r)$-module structure on $V_{n,r}$. 
For later use, we write  the explicit action of $\mathfrak{gl}_n$ on $V_{n,r}$ as follows: 
\begin{align}\label{acgln}
 E_{i,j}\mapsto \sum_{\ell=0}^{r-1} y_{i,\ell}\frac{\partial}{\partial y_{j,\ell}}\quad \text{for $1\leq i,j\leq n$.}
\end{align}

\begin{dfnt} Let $\xi$ and $\eta$ be two partitions with $\ell(\xi)\leq n$ and $\ell(\eta)\leq r$.
A non-zero vector $v$ in $V_{n,r}$ is called a highest weight vector for $\mathfrak{gl}_n\times \mathfrak{gl}_r$ with type $(\xi,\eta)$  if 
\begin{align*}
    (\mathfrak{n}_{n}^+\times  \mathfrak{n}_{r}^+).v=0\ \text{and}\ (E_{i,i},E_{j,j}).v=(\xi_i+\eta_j)v\ \text{for}\ 1\le i\le n, 1\le j\le r.
\end{align*}

\end{dfnt}
For every partition $\eta=(\eta_1,\eta_2,\dots)$, let $\eta'=(\eta_1',\eta_2',\dots)$
denote its transposed partition, where
\[\eta_i'=\textrm{Card}\{j\in\Z_{+}\mid \eta_j\geq i\}.\] 
 Write $\mathcal{P}_{n,r}$ for the set of all partitions $\eta$ such that $\ell(\eta)\le n$ and $\ell(\eta')\le r$.
Let $\xi=(\xi_1,\xi_2,\dots)$ be a partition in $\mathcal{P}_{n,r}$. Write 
\begin{align}\label{eq:yxi}
\bm{y}_\xi:=\bm{y}_1(\xi_1)\bm{y}_{2}(\xi_2)\cdots  \bm{y}_{\ell(\xi)}(\xi_{\ell(\xi)})\in V_{n,r}.
\end{align}
Here, for $1\leq i\leq \ell(\xi)$, \[\bm{y}_i(\xi_i):=y_{i,0}y_{i,1}\cdots y_{i,\xi_i-1}.\]

The structure of $(\mathfrak{gl}_n\times\mathfrak{gl}_r)$-module $V_{n,r}$ is completely determined by the highest weight vectors inside it. 
In fact, the skew $(\mathfrak{gl}_n,\mathfrak{gl}_r)$ Howe duality asserts that as a $(\mathfrak{gl}_{n}\times  \mathfrak{gl}_{r})$-module, 
$V_{n,r}$ has the following multiplicity-free decomposition:
\[V_{n,r}\cong \bigoplus_{\xi\in\mathcal{P}_{n,r}}L_{\mathfrak{gl}_{n}}(\xi)\otimes L_{\mathfrak{gl}_{r}}(\xi').\]
This decomposition can be deduced from  the following result (cf.\,\cite{H2}).
\begin{prpt}\label{jointh}
For every $\xi\in\mathcal{P}_{n,r}$, $\bm{y}_\xi$ is a highest weight vector  for $\mathfrak{gl}_n\times \mathfrak{gl}_r$ in $V_{n,r}$ with type $(\xi,\xi')$. Conversely, up to a scalar, every highest weight vector  for $\mathfrak{gl}_n\times \mathfrak{gl}_r$ in $V_{n,r}$ has such a form.  
\end{prpt}

\section{Local Weyl modules for loop algebra of $\mathfrak{sl}_n$}\label{sec:3}
In the rest of this paper, let $n$ and $r$ be two positive integers with $n\geq 2$.
In this section, we introduce some basics on local Weyl modules.

 As in the introduction, denote by 
\[\mathfrak{sl}_{n}=[\mathfrak{gl}_{n},\mathfrak{gl}_{n}]\] 
the complex simple Lie algebra of type $A_{n-1}$. Note that  $\eqref{trigl}$ induces a 
triangular decomposition of $\mathfrak{sl}_{n}$ as follows:
\begin{align}
\mathfrak{sl}_{n}={\mathfrak{n}}_{n}^+\oplus{\mathfrak{h}_n}\oplus{\mathfrak{n}}_{n}^-,
\end{align}  where  \[\mathfrak{h}_n:=\mathfrak{d}_n\cap \mathfrak{sl}_{n}=\bigoplus_{i=1}^{n-1}\C h_i \qquad (h_i:=E_{i,i}-E_{i+1,i+1}).\] 
Write $\{\omega_1,\omega_2,\dots,\omega_{n-1}\}\subset (\fh_{n})^{\ast}$ for  the dual basis of $\{h_1,h_2,\dots,h_{n-1}\}$. Namely, 
\[\omega_i(h_j)=\delta_{i,j}\quad\te{for }1\leq i,j\leq n-1.\]
 Then 
\begin{align}\label{eq:dom}
P_{+}=\bigoplus_{i=1}^{n-1}\N\omega_i\subset(\mathfrak{h}_n)^{\ast}
\end{align} 
is the set of  dominant integral weights of $\mathfrak{sl}_{n}$.

For every complex Lie algebra $\mathfrak{a}$, denote by 
 \[L(\mathfrak{a}):=\mathfrak{a}\otimes \C[t,t^{-1}]\]
the loop algebra of $\mathfrak{a}$, where the  Lie bracket  is given by
\begin{align*}
[x(m),y(n)]=[x,y](m+n)\quad\te{for }x,y\in  \mathfrak{a}\te{ and }m,n\in\Z.
\end{align*}
 Here and henceforth, we set \[x(m)=x\otimes t^m\in L(\mathfrak{a})\quad \te{for }x\in\mathfrak{a}\te{ and }m\in\Z.\]


For every polynomial $\pi(z)$ with  constant term 1, define \[\pi^{+}(z)=\pi(z)\quad \textrm{and}\quad \pi^{-}(z)=\frac{z^{\deg \pi(z)}\pi(z^{-1})}{(z^{\deg \pi(z)}\pi(z^{-1}))|_{z=0}} \in \C[z],\]
where $\deg \pi(z)$ denotes the degree of $\pi(z)$.
Write $\mathcal{M}_n$ for the  set of all 
 $(n-1)$-tuples  \[\bm{\pi}=(\pi_1(z),\pi_2(z),\dots,\pi_{n-1}(z)),\] 
 where each $\pi_i(z)$ is a polynomial with constant term $1$.
 The local Weyl modules for $L(\mathfrak{sl}_{n})$ are parametrized by the elements in $\mathcal{M}_n$ (see\,\cite{CP}).

\begin{dfnt}\label{defweyl}
Let $\bm{\pi}=(\pi_1(z),\pi_2(z),\dots,\pi_{n-1}(z))\in\mathcal{M}_n$. 
The local Weyl module of  $L(\mathfrak{sl}_{n})$ associated to $\bm{\pi}$, denoted by $W(\bm{\pi})$, is a cyclic $L(\mathfrak{sl}_{n})$-module generated by  $v$ and subject to  the defining  relations 
\begin{align}\label{rela1} L({\mathfrak{n}}_{n}^+).v=0,\quad  h_i(0).v=(\deg\pi_i(z) )v,\quad\Lambda_i^{\pm}(z).v=\pi_i^{\pm}(z) v,
\end{align}
and 
\[(E_{i+1,i}(0))^{(\deg \pi_i(z))+1}.v=0\quad \te{for all }1\leq i\leq n-1.
\]
Here, \[\Lambda_i^{\pm}(z):=\mathrm{exp}\left(-\sum_{ m\in\Z_{+}}\frac{h_{i}(\pm m)}{m}z^{m}\right) \in L(\mathfrak{h}_n)[[z]].\]
\end{dfnt}

For every $a\in\C^{\times}$ and $\lambda=\sum_{i=1}^{n-1}\lambda_i\omega_i\in P_{+}$, set
\begin{align}\label{pia}
    \bm{\pi}_{\lambda,a}:=((1-az)^{\lambda_1},\dots ,(1-az)^{\lambda_{n-1}})\in \mathcal{M}_n.
\end{align}
We remark that when $\bm\pi=\bm{\pi}_{\lambda,a}$, 
the second and third defining relations in \eqref{rela1} are equivalent to  
the relations 
\begin{align}\label{eq:anotherhaction}
h_i(m).v=a^m \lambda_i v\quad\text{for all  $m\in\Z$ and $1\leq i\leq n-1$.}
\end{align}


Finally, we recall the following factorization property of local Weyl modules for later use (see \cite[Theorem 2]{CP}).
\begin{prpt}\label{prop:tensordec}
    Let $\bm{\pi}=(\pi_1(z),\dots,\pi_{n-1}(z))$, $\tilde{\bm{\pi}}=(\tilde{\pi}_1(z),\dots,\tilde{\pi}_{n-1}(z))\in\mathcal{M}_n$ 
such that $\pi_i(z)$ and $\tilde{\pi}_j(z)$ are coprime for $1\leq i,j\leq n-1$. Then one has that
\[W(\bm{\pi\tilde{\pi}})\cong W(\bm{\pi})\otimes W(\bm{\tilde{\pi}})\]
as $L(\mathfrak{sl}_{n})$-modules, where
\[\bm{\pi\tilde{\pi}}=(\pi_1(z)\tilde{\pi}_1(z),\pi_2(z)\tilde{\pi}_2(z),\dots,\pi_{n-1}(z)\tilde{\pi}_{n-1}(z)).\]
\end{prpt}

\section{Main result}\label{sec:4}
In this section, we state the main result of this paper.
Recall  that $V_{n,r}$ (see \eqref{fs}) denotes the exterior algebra in the variables $y_{i,j}$ for  $1\leq i\leq n$ and $0\le j\le r-1$.

For $1\leq i\leq n$ and $0\leq j\leq r-1$, we set
 \[
x_{i,j} := \sum_{k=0}^{j} \binom{j}{k} \, y_{i, r-1-k}\in V_{n,r}.
\]
It is straightforward to see that 
\begin{align}\label{eq:xijpar}
    y_{i,j}=\sum_{k=1}^{r-j}\binom{r-j-1}{k-1}(-1)^{r-j-k}x_{i,k-1},
\end{align}
and so 
\[
V_{n,r}=\Lambda[x_{i,j}\mid 1\leq i\leq n,~0\leq j\leq r-1]
\]
is the  exterior algebra in the variables $x_{i,j}$ as well.

Let $a$ be a non-zero complex number. 
Define a linear map \[
\psi_{a,n,r}:L(\mathfrak{gl}_{n})\to \textrm{End}(V_{n,r})\]
by letting 
\begin{align}\label{eq:ac}
\psi_{a,n,r}(E_{i,k}(m))=a^m\sum_{j,s=0}^{r-1}\ell_{j,r}(m+s)x_{i,j}\frac{\partial}{\partial{x_{k,s}}}
\end{align}
for $1\leq i,k\leq n$ and $m\in\Z$. Here, \begin{align*}
\ell_{j,r}(t)=\prod_{0\le i\ne j\le r-1} \frac{t-i}{j-i}
\end{align*} denotes the fundamental Lagrange interpolation polynomial of degree $r-1$ at the points $j=0,1,\dots,r-1$.

\begin{prpt}\label{pr:homo}
The linear map $\psi_{a,n,r}$ is a Lie algebra homomorphism.
\end{prpt}
\begin{proof}
For  a polynomial $g(t)$ with $\te{deg\,}g(t)\leq r-1$, one has  (cf.\,\cite[Lemma 4.4]{CGT})
\begin{align}\label{interpoly}
g(m+n)=\sum_{j=0}^{r-1}g(m+j)\ell_{j,r}(n)\quad\textrm{for all }m,n\in\Z.
\end{align}
Then, for $1\leq i_1,i_2,k_1,k_2\leq n$ and $m_1,m_2\in\Z$, we have
\begin{align*}
    &[\psi_{a,n,r}(E_{i_1,k_1}(m_1)),\psi_{a,n,r}(E_{i_2,k_2}(m_2))]\\
    =\ &a^{m_1+m_2}\sum_{j_1,s_1,j_2,s_2=0}^{r-1}\ell_{j_1,r}(m_1+s_1)\ell_{j_2,r}(m_2+s_2)\left[x_{i_1,j_1}\frac{\partial}{\partial x_{k_1,s_1}},x_{i_2,j_2}\frac{\partial}{\partial x_{k_2,s_2}}\right]\\
    =\ &a^{m_1+m_2}\delta_{i_2,k_1}\sum_{j_1,s_2=0}^{r-1}\left(\sum_{j_2=0}^{r-1}\ell_{j_1,r}(m_1+j_2)\ell_{j_2,r}(m_2+s_2)\right)x_{i_1,j_1}\frac{\partial}{\partial x_{k_2,s_2}}\\
    &-a^{m_1+m_2}\delta_{i_1,k_2}\sum_{j_2,s_1=0}^{r-1}\left(\sum_{s_2=0}^{r-1}\ell_{j_2,r}(m_2+s_2)\ell_{s_2,r}(m_1+s_1)\right)x_{i_2,j_2}\frac{\partial}{\partial x_{k_1,s_1}}\\
    =\ &a^{m_1+m_2}\delta_{i_2,k_1}\sum_{j_1,s_2=0}^{r-1}\ell_{j_1,r}(m_1+m_2+s_2)x_{i_1,j_1}\frac{\partial}{\partial x_{k_2,s_2}}\\
    &- a^{m_1+m_2}\delta_{i_1,k_2}\sum_{j_2,s_1=0}^{r-1}\ell_{j_2,r}(m_1+m_2+s_1)x_{i_2,j_2}\frac{\partial}{\partial x_{k_1,s_1}}\\
    =\ &\psi_{a,n,r}\left([E_{i_1,k_1}(m_1),E_{i_2,k_2}(m_2)]\right),
\end{align*}
where in the second last equality we have used \eqref{interpoly} with $g(t)=\ell_{j_1,r}(t)$ and $\ell_{j_2,r}(t)$. 

\end{proof}

Via the homomorphism $\psi_{a,n,r}$, $V_{n,r}$ becomes an $L(\mathfrak{gl}_{n})$-module. Note that 
\begin{align*} \frac{\partial}{\partial x_{i,j}} = \sum_{k=1}^{r-j} \binom{r-k}{j} (-1)^{r-k-j} \frac{\partial}{\partial y_{i,k-1}}
 \end{align*}
 for $1\leq i\leq n$ and $0\le j\le r-1$.
This, together with  \eqref{eq:xijpar}, implies that 
\begin{align*}
\psi_{a,n,r}(E_{i,j}(0))=\ &  \sum_{k=0}^{r-1}x_{i,k}\frac{\partial}{\partial{x_{j,k}}}\\
=\ &  \sum_{k=0}^{r-1} x_{i,k}\left(\sum_{p=0}^{r-k-1} \binom{r-p-1}{k} (-1)^{r-k-p-1} \frac{\partial}{\partial y_{j,p}}\right)\\
=\ &\sum_{p=0}^{r-1}\left(\sum_{k=0}^{r-p-1}\binom{r-p-1}{k}(-1)^{r-p-k-1}x_{i,k}\right)\frac{\partial}{\partial y_{j,p}}\\
=\ & \sum_{p=0}^{r-1} y_{i,p}\frac{\partial}{\partial y_{j,p}}
\end{align*}
for $1\le i,j\le n$. 
Thus, the action of $L(\mathfrak{gl}_{n})$ on $V_{n,r}$ is an extension of the action of $\mathfrak{gl}_{n}$ on $V_{n,r}$  (see \eqref{acgln}). 

From now on, we view $V_{n,r}$ as an $L(\mathfrak{sl}_{n})$-module by taking restriction, denoted by $V_{n,r}(a)$.
Let $\xi=(\xi_1,\xi_2,\dots)$ be a partition in $\mathcal{P}_{n,r}$.
Denote by
\begin{equation}\label{lamxi}
\lambda(\xi):=
(\xi_1-\xi_2)\omega_1+(\xi_2-\xi_3)\omega_2+\dots+(\xi_{n-1}-\xi_{n})\omega_{n-1}
\end{equation} 
 the dominant integral weight of $\mathfrak{sl}_{n}$ associated to $\xi$. 
 Then we have the local Weyl module $W({\bm{\pi}}_{\lambda(\xi),a})$ for $L(\mathfrak{sl}_n)$ (see \eqref{pia}). 
On the other hand, recall that  there is a Grassmann monomial $\bm{y}_\xi$  in $V_{n,r}$ associated to $\xi$ (see \eqref{eq:yxi}). 
Write 
$W_{a,\xi}$ for the $L(\mathfrak{sl}_{n})$-submodule of $V_{n,r}(a)$ generated by the element $\bm{y}_\xi$. 

 The following  is the main result of this paper, whose proof will be  presented in Section \ref{sec:6}.
\begin{thmt}\label{th:main}
Let $a\in \C^{\times}$ and let $\xi\in \mathcal{P}_{n,r}$. Then  
\begin{align}
W_{a,\xi}\cong W(\bm{\pi}_{\lambda(\xi),a})
 \end{align}
as $L(\mathfrak{sl}_{n})$-modules.
 \end{thmt}

Finally, let $\bm{\pi}=(\pi_1(z),\pi_2(z),\dots,\pi_{n-1}(z))\in\mathcal{M}_n$.  For every $1\le i\le n-1$, we know
 \begin{align}\label{genpi}
 \pi_i(z)=\prod_{k=1}^{s_{i}}(1-b_{i,k}z)^{r_{i,k}}
 \end{align}
 for some (uniquely determined) non-zero distinct complex numbers $b_{i,1},b_{i,2},\dots,b_{i,s_i}$ and positive integers $r_{i,1},r_{i,2}, \dots,r_{i,s_i}$. 
 Let $a^{(1)}, a ^{(2)}, \dots, a^{(m)} $ be the distinct elements in the set $\{b_{i,k}\mid 1 \leq i \leq n-1, 1 \leq k \leq s_{i}\}, $ and let $\lambda^{(j)}=\sum_{i=1}^{n-1}\lambda_i^{(j)}\omega_i\in P_{+}$ for $1 \leq j \leq m$ such that
 \[\lambda_{i} ^{(j)} = 
 \begin{cases}
     r_{i,k} , &\mbox{if } a^{(j)} = b_{i,k} \mbox{ for some }k,\\
     0, &\mbox{otherwise}.
 \end{cases}\]
Then one has that
$\bm{\pi }= \displaystyle\prod_{j=1}^{m} \bm{\pi}_{\lambda^{(j)},a^{(j)}}$, and hence 
\begin{align}\label{tensor}
W(\bm\pi)\cong \otimes_{j=1}^m W(\bm{\pi}_{\lambda^{(j)},a^{(j)}})
\end{align}
by Proposition \ref{prop:tensordec}. 

Form the tensor $L(\mathfrak{sl}_n)$-module 
\[
V_{\bm{\pi}}:=\otimes_{j=1}^m V_{n,r^{(j)}}(a^{(j)}),
\]
where $r^{(j)}=\sum\limits_{k=1}^{n-1}\lambda_k^{(j)}$ for $1\leq j\leq m$. 
Note that
\[
W_{\bm{\pi}}:=\otimes_{j=1}^m W_{a^{(j)},\xi^{(j)}}
\]
is the submodule of $V_{\bm{\pi}}$ generated by the vector $\otimes_{j=1}^m \bm{y}_{\xi^{(j)}}$,
where 
\[\xi^{(j)}=(\sum_{k=1}^{n-1}\lambda_k^{(j)},\sum_{k=2}^{n-1}\lambda_k^{(j)},\dots,\lambda_{n-1}^{(j)},0,\dots)\in\mathcal{P}_{n,r^{(j)}}\quad\text{for $1\le j\le m$}.\]
Combining  Theorem \ref{th:main} with \eqref{tensor}, we have the following realization of all local Weyl modules for $L(\mathfrak{sl}_{n})$.
\begin{cort} \label{cor:main}With the notation given above, one has that
  \begin{align*}
    W(\bm\pi)\cong W_{\bm{\pi}}
\end{align*}
as $L(\mathfrak{sl}_{n})$-modules.

\end{cort}

\section{Local Weyl modules for Takiff algebra of $\mathfrak{sl}_{n}$}\label{sec:5}
In this section we give a realization of local Weyl modules for the $(r-1)$-th Takiff algebra of $\mathfrak{sl}_{n}$,
which plays a key role in the proof of Theorem \ref{th:main}. 

 For every complex Lie algebra $\mathfrak{a}$, write 
\[
T_r(
 \mathfrak{a}):=\mathfrak{a}\otimes \C[t]/t^{r}\C[t]\]
for the $(r-1)$-th Takiff algebra of $\mathfrak{a}$.
For every $i\in \N$, set $\overline{t^i}:=t^i+t^{r}\C[t]$. 
Then 
\[\{\overline{t^i}\mid 0\leq i\leq r-1\}\]
forms a basis of $\C[t]/t^{r}\C[t]$.


We remark that there is a natural embedding from $T_r(\mathfrak{gl}_{n})$ into $\mathfrak{gl}_{nr}$ defined by 
\[ E_{i,j}\otimes \overline{t^k}\mapsto \sum_{s=0}^{r-k-1}E_{i+sn,j+(s+k)n}\quad \text{for $1\leq i,j\leq n$ and $0\leq k\leq r-1$.}\]
This, together with \eqref{acglnr},
yields a $T_r(\mathfrak{gl}_{n})$-module structure on $V_{n,r}$ with the action given by
\begin{align}\label{actak}
E_{i,j}\otimes \overline{t^k}\mapsto \sum_{s=0}^{r-k-1} y_{i,s}\frac{\partial}{\partial y_{j,s+k}}\quad\text{for $1\leq i,j\leq n$ and $0\leq k\leq r-1$.}
\end{align}
By taking restriction, $V_{n,r}$ becomes a  $T_r(\mathfrak{sl}_{n})$-module. We denote by $T_r(W_\xi)$ the $T_r(\mathfrak{sl}_{n})$-submodule of $V_{n,r}$ generated by ${\bm y_\xi}$.

The following result is straightforward. 
\begin{lemt}\label{lem:trhighest}
 The following relations hold in $T_r(W_\xi):$
\begin{align*}
 T_r({\mathfrak{n}}_{n}^+).\bm y_{\xi}=0,\ (h_i\otimes \overline{t^j}).\bm y_{\xi}=\delta_{j,0}(\xi_i-\xi_{i+1})\bm y_{\xi},\ \textrm{and}\ (E_{i+1,i}\otimes \overline{1})^{\xi_i-\xi_{i+1}+1}.\bm y_{\xi}=0,
\end{align*}
where $1\leq i\leq n-1$ and $0\leq j\leq r-1$.  
\end{lemt} 


Similar to Definition \ref{defweyl}, 
one also has the notion of local Weyl modules for $T_r(\mathfrak{sl}_n)$ (cf.\,\cite{FMM}). For every 
 $\lambda\in P_{+}$, denote by 
$W_r(\lambda)$ the local Weyl module of $T_r(\mathfrak{sl}_n)$ associated to $\lambda$, which by definition is the cyclic $T_r(\mathfrak{sl}_{n})$-module generated by a vector $v$ and subject to  the defining  relations 
\begin{align}\label{rela2} T_r({\mathfrak{n}}_{n}^+).v=0,\quad  (h_i\otimes \overline{t^j}).v=\delta_{j,0}\lambda(h_i)v,\quad\textrm{and}\quad(E_{i+1,i}\otimes 1)^{\lambda(h_i)+1}.v=0
\end{align}
  for $1\leq i\leq n-1$ and $0\leq j\leq r-1$. 

The following result gives a realization of these local Weyl modules for $T_r(\mathfrak{sl}_{n})$.
\begin{thmt}\label{th:tak}
 Let $n,r\in\Z_{+}$ with $n\geq 2$  and  let $\xi\in \mathcal{P}_{n,r}$. Then 
 one has that
 \[T_r(W_\xi)\cong W_r(\lambda(\xi))\]
 as $T_r(\mathfrak{sl}_{n})$-modules.
\end{thmt}

Lemma \ref{lem:trhighest} implies that the $T_r(\mathfrak{sl}_n)$-module $T_r(W_\xi)$ is a quotient of $W_r(\lambda(\xi))$. In addition, 
it is known that  (see \cite[1.5.1]{CL} and \cite[Eq.~(2.2.5)]{FMM})
 \begin{align}\label{eq:dimwrlambdaxi}
     \dim W_r(\lambda(\xi))=\prod_{i=1}^{n-1}\binom{n}{i}^{\xi_i-\xi_{i+1}}.
 \end{align}
 Therefore, Theorem \ref{th:tak} follows from the following result.
 \begin{prpt}\label{prop:dimtak} One has that
 \begin{align}\label{dimtak}
     \dim T_r(W_\xi)\geq \prod_{i=1}^{n-1} \binom{n}{i}^{\xi_i-\xi_{i+1}}.
 \end{align}
 \end{prpt}
 
 The rest of this section is devoted to a proof of Proposition \ref{prop:dimtak}. We will prove the proposition by constructing $\prod_{i=1}^{n-1} \binom{n}{i}^{\xi_i-\xi_{i+1}}$ linearly independent vectors in $T_r(W_\xi)$. To this end, we first introduce some notations.  Let $``\preceq"$ denote the reversed lexicographic (total) order on $\mathbb{N}^k$ for $k\in\mathbb{N}$. More precisely, for 
\[
\bm i=(i_1,\dots,i_k),\quad \bm j=(j_1,\dots,j_k)\in \mathbb{N}^k,
\]
we write $\bm i\preceq \bm j$ if either $\bm i=\bm j$, or the last non-vanishing difference $i_s-j_s$ ($1\le s\le k$) is negative. In the latter case, we write $\bm i\prec \bm j$.

Let $m,p,q\in\N$ with $p\le q$. Denote by  $\mathbb{N}_{\mathrm{ord}}^m$ the subset of $\N^m$ consisting of strictly increasing $m$-tuples,
and by $\mathbb{N}_{\mathrm{ord}}^m(p,q)\subseteq \mathbb{N}_{\mathrm{ord}}^m$ the set consisting of those $\bm k=(k_1,\dots,k_m)$ satisfying the condition 
\[
p\le k_1<\cdots<k_m\le q.
\]
When $m>0$ and $q=p+m-1$, we have  
\[\mathbb{N}_{\mathrm{ord}}^m(p,q)=\{(p)_m\}\quad ((p)_m:=(p,p+1,\dots,p+m-1)).\] 
Assume further that $1\leq m\leq q-p+1$. For every $\bm{d}=(d_1,d_2,\dots,d_{m})\in \N^{m}$, we define $S(p,q,\bm{d})$ 
to be the subset of $\mathbb{N}_{\mathrm{ord}}^{m}$ consisting of all $\bm u=(u_1,u_2,\dots,u_{m})$
such that 
\[(u_1+d_{\sigma(1)},u_2+d_{\sigma(2)},\dots,u_{m}+d_{\sigma(m)})\in \mathbb{N}_{\mathrm{ord}}^{m}(p,q)\quad \te{for some }\sigma\in S_m.\]



 \begin{lemt}\label{le:maxord}
 Assume that $d_1\leq d_2\leq \dots \leq d_m\leq q-p-m+1$. Then 
 \[(q-m+1-d_{m},\dots,q-1-d_2,q-d_1)\]
is the unique maximal element in $S(p,q,\bm{d})$ with respect to the  order $``\preceq"$ on $\N^{m}$. 
 \end{lemt}
\begin{proof}
The assertion follows from the fact that 
\[(q-m+1)_m=(q-m+1,q-m+2,\dots,q)\]
is the unique maximal element in $\mathbb{N}_{\mathrm{ord}}^{m}(p,q)$  with respect to  $``\preceq"$.
\end{proof}

Let $1\le i\le n$, $0\le m\le r$, and $\bm j=(j_1,\dots,j_m)\in \mathbb{N}_{\mathrm{ord}}^{m}(0,r-1)$.
Set $\bm y_{i,\bm j}=1$ if $m=0$, and 
\[ \bm y_{i,\bm j}:=y_{i,j_1}\cdots y_{i,j_m}\in V_{n,r}\]
if $m>0$. 
Then the set
\[
\{\bm y_{1,\bm j_1}\bm y_{2,\bm j_2}\cdots \bm y_{n,\bm j_{n}}\mid \bm j_t\in \cup_{m=0}^{r}\mathbb{N}_{\mathrm{ord}}^{m}(0,r-1),~1\leq t\le n\}
\]
forms a basis of $V_{n,r}$.

We first consider the case that $n=2$. Introduce the  elements 
\begin{align*}
Y(k,\bm{m})=\left(\prod_{i=1}^{k}E_{2,1}\otimes \overline{t^{\xi_1-m_{k+1-i}-i}}\right).\bm{y}_{\xi}\in T_r(W_\xi),
\end{align*}
where $0\leq k\leq \xi_1-\xi_2$ and $\bm{m}=(m_1,\dots,m_{k})\in\N_{\textrm{ord}}^k(\xi_2,\xi_1-1)$. These elements are well-defined since $[E_{2,1}\otimes \overline{t^i}, E_{2,1}\otimes \overline{t^j}]=0$ for all $0\leq i,j\le r-1$.

\begin{prpt}\label{sl2}
Let $\xi=(\xi_1,\xi_2)\in\mathcal{P}_{2,r}$. Then the elements $Y(k,\bm{m})$ for $0\leq k\leq \xi_1-\xi_2$ and $\bm{m}=(m_1,\dots,m_{k})\in\N_{\mathrm{ord}}^k(\xi_2,\xi_1-1)$ are linearly independent in $T_r(W_\xi)$.
\end{prpt}
\begin{proof}
It is clear that the assertion holds for the case that $\xi_1=\xi_2$. Now assume that $\xi_1\neq \xi_2$.
Let $0\leq k\leq \xi_1-\xi_2$ and $\bm{m}=(m_1,\dots,m_{k})\in\N_{\textrm{ord}}^k(\xi_2,\xi_1-1)$. 
For $1\leq i\leq k$, we set $d_i=\xi_1-m_{k+1-i}-i$ and write
$\bm{d}=(d_1,d_2,\dots,d_k)$. Then one easily verifies  that 
$d_1\leq d_2\leq\dots\le d_k\leq \xi_1-\xi_2-k.$

For every $\bm u\in \N_{\textrm{ord}}^k(\xi_2,\xi_1-1)$ and $\sigma\in S_k$, we have that 
\[\bm{u}\preceq (m_1+d_{\sigma(k)}, m_2+d_{\sigma(k-1)},\dots ,m_k+d_{\sigma(1)}).\]
Furthermore, the above is an equality  if and only if $\bm{u}=(\xi_1-k,\xi_1-k+1,\dots,\xi_1-1)$ and $d_{\sigma(i)}=d_i$ for $1\leq i\leq k$. This, together with \eqref{actak} and  Lemma \ref{le:maxord}, implies that 
\begin{align*}
&Y(k,\bm m)\\
=\ &\sum_{\bm p\in S(\xi_2,\xi_1-1,\bm{d})}\sum_{\sigma\in S_k}\left(y_{2,p_1}\frac{\partial}{\partial y_{1,p_1+d_{\sigma(k)}}}y_{2,p_2}\frac{\partial}{\partial y_{1,p_2+d_{\sigma(k-1)}}}\dots y_{2,p_k}\frac{\partial}{\partial y_{1,p_k+d_{\sigma(1)}}}\right).\bm{y}_{\xi}\\
 =\ &c_{\bm m} \bm{y}_{1,(0)_{\xi_1-k}}\bm{y}_{2,(0)_{\xi_2}}\bm{y}_{2,\bm{m}}+ \sum_{\substack{\bm{q}\in \N_{\textrm{ord}}^{\xi_1-k}(0,\xi_1-1),\\ 
\bm{p}\prec \bm{m}}}c_{\bm{q},\bm{p}}\bm{y}_{1,\bm q}\bm{y}_{2,(0)_{\xi_2}}\bm{y}_{2,\bm{p}},
\end{align*}
where $c_{\bm{q},\bm{p}}\in\C$ and $c_{\bm m}\in\C^{\times}$.  The proposition  then follows immediately.
\end{proof}

\noindent\textbf{Proof of Proposition \ref{prop:dimtak}:}
We prove the theorem by using induction on $n$. When $n=2$, it follows from Proposition \ref{sl2} that 
\[\textrm{dim }T_r(W_\xi)\geq \sum_{k=0}^{\xi_1-\xi_2}\textrm{Card\ } \N_{\textrm{ord}}^k(\xi_2,\xi_1-1)=\sum_{k=0}^{\xi_1-\xi_2}\binom{\xi_1-\xi_2}{k}=2^{\xi_1-\xi_2},\]
as desired.  

Now assume that $n>2$. We set 
\[\N(\xi)=\{\bm{k}=(k_1,k_2,\dots,k_{n-1})\in\N^{n-1}\mid 0\leq k_i\leq \xi_i-\xi_{i+1},~1\leq i\leq n-1\},\]
and introduce the following elements
\begin{align*}
Y(\bm{k},\bm{m}_1,\dots,\bm{m}_{n-1})=\prod_{j=1}^{n-1}\left(\prod_{i=1}^{k_j}E_{n,j}\otimes \overline{t^{\xi_j-m_{j,k_j-i+1}-i}}\right).\bm{y}_{\xi}\in T_r(W_\xi),
\end{align*}
where $\bm{k}\in \N(\xi)$ and $\bm{m}_i=(m_{i,1},\dots,m_{i,k_i})\in\N_{\te{ord}}^{k_i}(\xi_{i+1},\xi_i-1)$ for $1\leq i\leq n-1$.

We claim that these elements are linearly independent in $T_r(W_\xi)$. Indeed, by a similar argument as that in Proposition \ref{sl2}, we have 
\begin{align}\label{basis2}
Y(\bm{k},\bm{m}_1,\dots,\bm{m}_{n-1})\nonumber=& \ c\bm{y}_{\bm\xi_n-\bm k}\bm{y}_{n,(0)_{\xi_{n}}}\bm{y}_{n,\bm{m}_{n-1}}\bm{y}_{n,\bm{m}_{n-2}}\cdots \bm{y}_{n,\bm{m}_1}\nonumber\\
&+\sum_{\substack{\bm{j}\in\N_{\textrm{ord}}^{\xi_{n}+\sum_{i=1}^{n-1}k_i}(0,r-1),\\ \bm{j}\prec ((0)_{\xi_{n}},\bm{m}_{n-1},\bm{m}_{n-2},\dots,\bm{m}_1)}}u_{\bm j}{\bm{y_{n,j}}},
\end{align}
where $c\in\C^{\times}$, $u_{\bm j}\in V_{n-1,r}(\subseteq V_{n,r})$, and 
\[\bm{\xi}_n-\bm k=(\xi_{1}-k_1,\xi_{2}-k_2,\dots,\xi_{n-1}-k_{n-1})\in\mathcal{P}_{n-1,r}.\]
This  implies the claim.

On the other hand, view $V_{n,r}$ as a $T_r(\mathfrak{sl}_{n-1})$-module by  restricting to $T_r(\mathfrak{sl}_{n-1})$ and consider the $T_r(\mathfrak{sl}_{n-1})$-submodule generated by $\bm{y}_{\bm\xi_n-\bm k}$.
By the  induction hypothesis, we have 
\begin{align}\label{hypo}
\dim \mathcal{U}(T_r(\mathfrak{sl}_{n-1})).\bm y_{\bm\xi_n-\bm k}\geq \prod_{i=1}^{n-2}\binom{n-1}{i}^{\xi_i-k_i-\xi_{i+1}+k_{i+1}}.
\end{align}
In addition, $T_r(\mathfrak{sl}_{n-1})$ acts trivially on $y_{n,j}$ for $0\leq j\leq r-1$. This, together with \eqref{basis2} and \eqref{hypo}, gives that 
\begin{eqnarray*}
\dim T_r(W_\xi)\geq&&
\sum_{\bm k\in\N(\xi)}\left(\prod_{i=1}^{n-2}\binom{n-1}{i}^{\xi_i-k_i-\xi_{i+1}+k_{i+1}}\right)\left(\prod_{j=1}^{n-1}\binom{\xi_j-\xi_{j+1}}{k_j}\right)\\
=&&\prod_{i=1}^{n-1}\left(\sum_{k_i=0}^{\xi_{i}-\xi_{i+1}}\binom{\xi_{i}-\xi_{i+1}}{k_i} \binom{n-1}{i}^{\xi_{i}-\xi_{i+1}-k_i}\binom{n-1}{i-1}^{k_{i}}\right)\\
 =&&\prod_{i=1}^{n-1}\binom{n}{i}^{\xi_{i}-\xi_{i+1}},
 \end{eqnarray*}
where in the last equality we used the fact that
\[\binom{n}{i}=\binom{n-1}{i}+\binom{n-1}{i-1}.\]
This completes the proof. \qed

\section{Proof of Theorem \ref{th:main}}\label{sec:6}
 This section is devoted to a proof of Theorem \ref{th:main}. 
We start with the following technical result (cf.\,\cite{CGT}).
\begin{lemt}\label{le:ann}
Let $m$ and $s$ be two non-negative integers with $m<s$, and let $g(t)$ be a polynomial  of degree $m$. 
Then
\[\sum\limits_{k=0}^s\binom{s}{k}(-1)^{s-k}g(k)=0\quad\text{and}\quad \sum\limits_{k=0}^s\binom{s}{k}(-1)^{s-k}k^s=s!.\]
\end{lemt}
Write $I_{r,a}$ for the ideal of $\C[t,t^{-1}]$ generated by $(t-a)^{r}$. 

\begin{lemt}\label{le:ker}
One has that $\mathfrak{sl}_{n}\otimes I_{r,a}\subseteq \ker \psi_{a,n,r}$.
\end{lemt}
\begin{proof}
By Lemma \ref{le:ann} and \eqref{eq:ac}, we have that
\begin{align*}
& \psi_{a,n,r}(E_{i,k}\ot (t-a)^{r})\\
=\ &(-a)^{r}\sum_{j_1,s_1=0}^{r-1}\left(\sum_{j_2=0}^{r-1}(-1)^{r-j_2-1}\binom{r-1}{j_2}\ell_{j_1,r}(r-j_2+s_1-1)\right)x_{i,j_1}\frac{\partial}{\partial x_{k,s_1}}\\
=\ &0,
\end{align*}
where $1\leq i,k\leq n$. 
It follows that \[\mathfrak{sl}_{n}\otimes I_{r,a}=[\mathfrak{gl}_{n},\mathfrak{gl}_{n}]\otimes I_{r,a}\subseteq \ker \psi_{a,n,r},\]
as desired.
\end{proof}

In view of Lemma \ref{le:ker}, the $L(\mathfrak{sl}_{n})$-module $V_{n,r}(a)$ descends to an 
$\mathfrak{sl}_{n}\otimes\left(\C[t,t^{-1}]/I_{r,a}\right)$-module. 
Then, via the canonical algebra isomorphism 
\[T_r(\mathfrak{sl}_{n})\to \mathfrak{sl}_{n}\otimes\left(\C[t,t^{-1}]/I_{r,a}\right), \quad x\otimes \overline{t^k}\mapsto x\otimes ((t-a)^k+I_{r,a}),\]
$V_{n,r}(a)$ becomes a $T_r(\mathfrak{sl}_{n})$-module, denoted by 
$\overline{V_{n,r}(a)}$. 

\begin{lemt}\label{lem:actionT-r}
     The action of $T_r(\mathfrak{sl}_{n})$ on $\overline{V_{n,r}(a)}$ is given by 
     \[
     E_{i,k}\ot \overline{t^m}\mapsto a^m \sum_{p=m}^{r-1}y_{i,p-m}\frac{\partial}{\partial y_{k,p}}\quad \text{for $1\leq i,k\leq n$ and $0\leq m\leq r-1$.}
     \]
  In particular, when $a=1$, it coincides with    \eqref{actak}.
\end{lemt}
\begin{proof}
By using Lemma \ref{le:ann} and \eqref{eq:xijpar}, we obtain that 
\begin{align*}\label{eq:opdes}
&E_{i,k}\ot \overline{t^m}\\
=\ &a^m\sum_{j_1,s_1=0}^{r-1}\left(\sum_{j_2=0}^{m}(-1)^{m-j_2}\binom{m}{j_2}\ell_{j_1,r}(j_2+s_1)\right)x_{i,j_1}\left(\sum_{p=0}^{r-s_1-1}\binom{r-p}{s_1}(-1)^{r-p-s_1-1}\frac{\partial}{\partial y_{k,p}}\right)\nonumber\\
=\ &a^m\sum_{j_1,p=0}^{r-1}\sum_{s_1=0}^{r-p-1}\left(\sum_{j_2=0}^{m}(-1)^{m-j_2+r-p-s_1-1}\binom{m}{j_2}\binom{r-p-1}{s_1}\ell_{j_1,r}(j_2+s_1)\right)x_{i,j_1}\frac{\partial}{\partial y_{k,p}}\nonumber\\
=\ &a^m\sum_{j_1,p=0}^{r-1}\sum_{s_1=0}^{m+r-p-1}\left((-1)^{m+r-p-s_1-1}\sum_{j_2=0}^{m}\binom{m}{j_2}\binom{r-p-1}{s_1-j_2}\ell_{j_1,r}(s_1)\right)x_{i,j_1}\frac{\partial}{\partial y_{k,p}}\nonumber\\
=\ & a^m\sum_{j_1,p=0}^{r-1}\sum_{s_1=0}^{m+r-p-1}\left((-1)^{m+r-p-s_1-1}\binom{m+r-p-1}{s_1
}\ell_{j_1,r}(s_1)\right)x_{i,j_1}\frac{\partial}{\partial y_{k,p}}\nonumber\\
=\ &a^m\sum_{p=0}^{r-1}\sum_{s_1=0}^{m+r-p-1}\left((-1)^{m+r-p-s_1-1}\binom{m+r-p-1}{s_1
}x_{i,s_1}\right)\frac{\partial}{\partial y_{k,p}}\nonumber\\
=\ &a^m\sum_{p=m}^{r-1}y_{i,p-m}\frac{\partial}{\partial y_{k,p}}\nonumber
\end{align*}
as operators on $\overline{V_{n,r}(a)}$.
\end{proof}

By applying Lemma \ref{lem:actionT-r}, the following is straightforward.

\begin{lemt}\label{lem:weylquotient}
The following relations hold in the $L(\mathfrak{sl}_n)$-module $V_{n,r}(a)$:
    \[
   L({\mathfrak{n}}_{n}^+).\bm y_{\xi}=0,\ h_i(m).\bm{y}_{\xi}=a^m(\xi_i-\xi_{i+1})\bm y_{\xi},\ \textrm{and}\ (E_{i+1,i}\otimes 1)^{\xi_i-\xi_{i+1}+1}.\bm y_{\xi}=0,
    \] where $1\leq i\leq n-1$ and $m\in \Z$.
\end{lemt}

\vspace{3mm}

\noindent\textbf{Proof of Theorem \ref{th:main}: } 
 Lemma \ref{lem:weylquotient} implies that the $L(\mathfrak{sl}_{n})$-module
$W_{a,\xi}$ is a quotient of the local Weyl module $W(\bm\pi_{\lambda(\xi),a})$ (see Definition \ref{defweyl} and \eqref{eq:anotherhaction}). 
Thus, it suffices to prove that 
\[
\dim W_{a,\xi}=\dim W(\bm\pi_{\lambda(\xi),a}).
\] 
On the one hand, it was shown in  \cite[Section 1.5.1]{CL} that 
\[\dim W(\bm\pi_{\lambda(\xi),a})=\prod_{i=1}^{n-1}\binom{n}{i}^{\xi_i-\xi_{i+1}}.\]
On the other hand, by twisting the automorphism 
\[
L(\mathfrak{sl}_n)\rightarrow L(\mathfrak{sl}_n),\quad x(m)\mapsto a^m x(m),
\]
there is a new  $L(\mathfrak{sl}_n)$-module structure on  $V_{n,r}(1)$, which is exactly $V_{n,r}(a)$. 
This, together with Lemma \ref{lem:actionT-r}, Theorem \ref{th:tak}, and \eqref{eq:dimwrlambdaxi}, gives that 
\[
\dim W_{a,\xi}=\dim W_{1,\xi}=\dim T_r(W_\xi)= \dim W_r(\lambda(\xi))=\prod_{i=1}^{n-1}\binom{n}{i}^{\xi_i-\xi_{i+1}}.
\]
Therefore, we complete the proof. \qed


\begin{thebibliography}{AAGBP}
\bibitem[BN]{BN} J. Beck and H. Nakajima, Crystal bases and two-sided cells of quantum affine algebras, Duke Math. J.
123 (2004), 335402.
\bibitem[CLS]{CLS}
L. Calixto, J. Lemay, and A. Savage, Weyl modules for Lie superalgebras, Proc. Amer. Math. Soc. 147 (2019),
3191–3207.
\bibitem[CFK]{CFK}
V. Chari, G. Fourier, and T. Khandai, A categorical approach to Weyl
modules, Transform. Groups 15 (2010), 517–549.
\bibitem[CFS]{CFS} 
 V. Chari, G. Fourier, and P. Senesi, Weyl modules for the twisted loop algebras, J. Algebra 319 (2008),
5016–5038.


 \bibitem[CP]{CP} V. Chari and A. Pressley,  Weyl modules for classical and quantum affine algebras, Represent. Theory 5 (2001), 191–223.
 \bibitem[CL]{CL} V. Chari and S. Loktev, Weyl, Demazure and fusion modules for the current algebra of $\mathfrak{sl}_{r+1}$, Adv. Math. 207 (2006), 928–960.
 \bibitem[CGT]{CGT}
 F. Chen, Y. Gao, and S. Tan, Realizations of $A_1^{(1)}$-modules in category $\widetilde{\mathcal{O}}$, Represent. Theory  27 (2023), 149–176.
 \bibitem[FM]{FM}
  E. Feigin and I. Makedonskyi, Generalized Weyl modules, alcove paths and Macdonald polynomials,
Selecta. Math. (N.S.) 23 (2017), 2863–2897.
\bibitem[FMO]{FMO}
E. Feigin, I. Makedonskyi, and D. Orr,
Generalized Weyl modules and nonsymmetric $q$-Whittaker functions,
Adv. Math. 330 (2018), 997–1033.
 \bibitem[FL]{FL}
 G. Fourier and P. Littelmann, Weyl modules, Demazure modules, KR-modules, crystals, fusion products and limit
constructions, Adv. Math. 211 (2007), 566–593.
\bibitem[FMM]{FMM}
 G. Fourier, V. Martins, and A. Moura, On truncated Weyl modules, Commun. Algebra 47 (2019), 1125–1146.
 \bibitem[H1]{H1}
 R. Howe, $\theta$-series and invariant theory, in: Proceedings of Symposium on Pure Mathematics, vol. 33,
AMS, Providence, (1979), 275–285.
\bibitem[H2]{H2}
R. Howe, Perspectives on invariant theory: Schur duality, multiplicity-free actions and
beyond, Israel Math. Conf. Proc. 8 (1995), 1–182.
\bibitem[K]{K} M. Kashiwara, Crystal bases of modified quantized enveloping algebra, Duke Math. J. 73 (1994), 383413.
 \bibitem[L]{L}
 M. Lau, Takiff algebras, Toda systems, and jet transformations, J. Algebra  655 (2024), 619-650.
\bibitem[MPS]{MPS}
S. Mukherjee, S. K. Pattanayak, and S.S. Sharma. Weyl
modules for toroidal Lie algebras, Algebr. Represent. Theory 26 (2023), 2605–2626.
\bibitem[N]{N}
K. Naoi, Weyl modules, Demazure modules and finite crystals for non-simply laced type, Adv.
Math. 229 (2012), 875–934.
\end{thebibliography}
\end{document}